\documentclass[twocolumn]{svjour3}

\usepackage[utf8]{inputenc}
\usepackage{ae}
\usepackage{amsmath}
\usepackage{graphicx}

\usepackage{color}

\newtheorem{theo}{Theorem}
\newtheorem{lem}[theo]{Lemma}

\newtheorem{remarque}[theo]{Remark}

\newcommand{\dpart}[3][]{\frac{\partial^{#1} #2}{\partial #3^{#1}}}

\begin{document}

\journalname{Computational Mechanics}
\titlerunning{Plate eigenmodes with Vekua approximations and the Method of Particular Solutions}
\title{Low-complexity  computation of plate eigenmodes with Vekua approximations and the Method of Particular Solutions}

\author{Gilles Chardon \and Laurent Daudet}

\institute{G. Chardon \at 
Institut Langevin - UPMC Univ. Paris 06\\
1 rue Jussieu\\
F-75005 Paris France\\
\email{gilles.chardon@m4x.org}\\
\emph{Present address:}
 Acoustics Research Institute - Austrian Academy of Sciences,\\
Wohllebengasse 12-14\\
A-1040 Wien Austria\\
Tel. +43 1 51581-2521
\and
L. Daudet \at
Institut Langevin\\
Paris Diderot University and Institut Universitaire de France\\
1 rue Jussieu\\
F-75005 Paris France}

\date{Received: date / Accepted: date}
\maketitle

\begin{abstract}

This paper extends the Method of Particular Solutions (MPS) to the computation of eigenfrequencies and eigenmodes 
of {thin plates, in the framework of the Kirchhoff-Love plate theory}. Specific approximation schemes are developed, 
 with plane waves (MPS-PW) or Fourier-Bessel functions (MPS-FB).  
This framework also requires a suitable formulation of the boundary conditions. 
Numerical tests, on two plates with various boundary conditions, demonstrate that the proposed approach provides competitive results with standard numerical schemes such as the Finite Element Method, at reduced complexity, and with large flexibility in the implementation choices.  
\end{abstract}

\section{Introduction}

Numerical computation of eigenfrequencies and eigenmodes of plates
is an important problem in mechanics. {We focus here on thin plates,
modeled by the Kirchhoff-Love plate theory. In this framework}, an eigenmode is a non-zero solution to
\begin{equation}
D \Delta^2 u +T \Delta w -  \rho h \omega^2 u = 0
\label{bilap}
\end{equation}
with boundary conditions, where $D$ is the rigidity of the plate,
$\rho$ the specific mass of its material, $h$ its thickness, and
$T$ the normal tension applied at its edges, assumed to be uniform.
The eigenfrequencies  are the frequencies $\omega$ such that a non-zero solution
exists.
Apart from particular cases where these 
quantities can be analytically computed (e.g., circular plates with
simple boundary conditions), they must be obtained by numerical methods.

The finite element method (FEM), which uses piecewise
polynomials to approximate the solutions, can be used to compute
these eigenmodes and eigenfrequencies. However, it can be computationally intensive
at high frequencies,
as the size of the numerical problem scales as 
the square of the spatial frequency{, or in cases where singularities
appear (e.g. polygonal plates). These singularities can be treated
by adaptive methods, that require more complex theory and
implementation \cite{Babuska}. 
Because finite element methods look for solutions in the entire space (e.g. $H^1$ for the case of eigenmodes of the Laplace operator),
the number of elements needed can be quite large, particularly
at high frequencies.}

{To deal with this problem, alternative methods to the FEM have
been developed, where only the solutions to the Helmholtz
equation with a particular wavenumber are approximated,
accelerating the convergence.
Within a family (different for each wavenumber) of such solutions, eigenmodes are the
linear  combinations that also satisfy the boundary conditions.
In the case of thin plates, one can cite} the boundary element
method (BEM) \cite{chen}, the method of fundamental solution (MFS) \cite{alves},
or its variant proposed by Kang \textit{et al.} \cite{KangLee}, the Non Dimensional
Influence Function (NDIF). {Eigenmodes are identified 
by computing the determinant of the operator mapping the 
vectors of size $n$ containing the coefficients of the expansion to the value
of the function at $n$ points
on the border. 
A determinant equal  (or near) to zero indicates a solution to Eqn.~(\ref{bilap}) that approximately satisfies the boundary
conditions, with non-zero coefficients.
In practice, local minima
of the function mapping the wavenumber to the determinant are assumed
to be eigenfrequencies.
 Implementation of these
methods can however be delicate, 
because of the appearance of
spurious modes, the difficult choice of the charge points
for the MFS and NDIF, or severe ill-conditioning of the numerical
problems at high frequencies or with large approximation orders.}

 Here, we investigate a new computational method derived from 
the Method of Particular Solutions, proposed by Fox, Henrici and Moller (FHM) \cite{FHM}, {analyzed by Eisenstat \cite{eisenstat}},
and further improved by Betcke and Trefethen \cite{Betcke}, for the computation of eigenmodes of the
Laplace operator. While the fundamental ideas used in this method are similar to
the previously cited methods, it has several advantages:
\begin{itemize}
\item the stability of the numerical problems is improved,
\item multiple eigenvalues are
easier to determine,
\item and basis of the associated eigenspaces are readily available.
\end{itemize}
{These improvements  are obtained by 
a change of criteria for an eigenmode being
non-zero : we here look at 
its $L_2$-norm, instead of its expansion coefficients. This means, implementation-wise, sampling not only the border of the domain, but also its interior,
and replacing the computation of a determinant by the search of the
largest eigenvalue of a generalized eigenvalue problem.}

\medskip
{The rest of this paper is constructed as follows. Section 
\ref{sec:mps} presents the MPS framework as introduced by 
Fox, Henrici and Moller, and some of its more recent developments.  We then extend this theory to the computation
of eigenmodes and eigenvalues of plates, assumed to be  star-shaped plates with smooth
boundaries - due to  fundamental limitations of the
Vekua theory.}
Our first contribution is the analysis of an
approximation scheme based on the Vekua theory, given in Section \ref{sec:app}. 
This provides some bounds on the approximation error
of a solution of Eqn. (\ref{bilap}) by sums of Fourier-Bessel functions
in Sobolev norms, based on similar results for the simpler case of the Helmholtz equation. Our second contribution,  described in Section \ref{sec:form},  is the formulation of the
problem in a way compatible with the MPS, and its numerical evaluation presented in Section \ref{sec:num}.
We discuss the extension to more general shapes, and various
implementation matters, in Section \ref{sec:discus}.

\section{The method of particular solutions}
\label{sec:mps}

The method of particular solutions was introduced by 
Fox, Henrici and Moller (FHM) \cite{FHM} for
the case of eigenmodes of the Laplace operator in a L-shaped domain with
a singular corner.

The basic idea of this method is, instead
of considering the entire space in which the eigenmodes are searched 
(e.g. $H^1$ for the case of the Laplace operator, approximated
by finite element spaces), to consider separately the spaces
of solutions of the Helmholtz equation for different wave numbers.
Then, in each of these spaces, we can look for a nonzero function
which also satisfies the boundary conditions, \textit{i.e.}, an eigenmode.
While building approximations for a lot of different spaces
seem to be counter-productive compared to the unique approximation
needed for a Galerkin approximation, this alternative scheme is 
interesting as efficient approximations can be obtained for these spaces
of solutions to the Helmholtz equation.

The method developed by FHM, for the computation
of eigenmodes of the Laplace operator with Dirichlet boundary conditions,
is as follows. For each wavenumber $k$, we consider
$N$ points $x_j$ on the border of the domain, and 
a family of $N$ functions $\phi_i$ spanning a subspace which approximates
the set of solutions
to the Helmholtz equation.
The considered family, Fourier-Bessel functions of fractional orders,
was specifically constructed to take into account the singularity arising
in the reentrant corner of the L-shaped domain. In order to find
the eigenfrequencies, one has to construct a square matrix $M(k)$ that contains
the values of the functions $\phi_j$ at the $N$ points of the border,
and to compute its determinant $d(k)$:
\begin{equation}
d(k) = \det M(k) = 
\left|
\begin{array}{ccc}
\phi_1(x_1) & \cdots & \phi_N(x_1) \\
\vdots & & \vdots \\
\phi_1(x_N) & \cdots & \phi_N(x_N)
\end{array}
\right|
\end{equation}
 If $k$ is an eigenfrequency, there is a non-zero 
 solution to the Helmholtz equation with values zero on the boundary,
and its approximation $\sum \alpha_i \phi_i$ is thus close to zero
at the sampling points, with nonzero coefficients $\alpha_i$.
The image of the vector of coefficients $(\alpha_i)$ by the matrix
 $M(k)$ is precisely the vector of the values of $\sum \alpha_i \phi_i$
on the points of the border. This means that the determinant of the matrix $M(k)$ is
 close to zero. Therefore, the eigenfrequencies are obtained as local minima
of $d(k)$.

Numerous variants of this method, using different approximation schemes, have
been developed since the original article. 
They mostly differ by the functions
used to approximate the solutions: the MFS uses fundamental solutions, the NDIF
\cite{KangLee}
uses Bessel functions of the first kind of order 0, etc.

As pointed out in \cite{Betcke}, this simple method has known limitations. 
The discretization of the space of solutions and the sampling of the boundaries
must be of the same size, and more importantly, the matrix
$M(k)$ gets ill-conditioned as the number $N$ of functions  grows.
 Indeed, using a larger family of functions $\phi_i$, in order  to have better approximations of
the modes, makes the problem numerically unstable. An interpretation of this fact
is that the algorithm does not search for a non-zero function inside
the domain with value zero on its boundary, but actually for a function with non-zero
coefficients of its expansion, with zero value on the boundary. The properties
of the approximating families are such that having non-zero expansion
coefficients does not ensure significant values of the function inside
the domain. 

To avoid these problems, Betcke and Trefethen \cite{Betcke} suggest to solve,
for each $k$, the following optimization problem :
\begin{equation}
\label{formBT}
\begin{split}
\tau(k) = \min_{u} \|Tu\|^2_{L_2(\partial \Omega)} \\
\mbox{ such that } \|u\|^2_{L_2(\Omega)} = 1 \mbox{ and } \Delta u + k^2u = 0
\end{split}
\end{equation}
where $T$ is the trace operator on the boundary $\partial \Omega$.
Then, $k$ is an eigenfrequency if and only if $\tau(k)$, called the tension,
is zero.

This problem can be discretized as follows. A family of functions
$(\phi_i)$ is chosen for the approximation of the solutions
of the Helmholtz equation. For a function with expansion coefficients $\mathbf u = (u_1,\ldots, u_{N})$,
$$\|Tu\|^2_{L_2(\partial \Omega)} =  \mathbf u^\star \mathbf F \mathbf u$$
$$\|u\|^2_{L_2(\Omega)} =\mathbf u^\star \mathbf G \mathbf u$$
where the coefficients of the matrices $\mathbf F$ and $\mathbf G$ are
\begin{equation}
F_{ij} = \left< T\phi_i, T\phi_j \right>_{L_2(\partial \Omega)},\ \ \ \ 
G_{ij} = \left< \phi_i, \phi_j \right>_{L_2(\Omega)}
\end{equation}
These scalar products can be estimated by sampling the domain and its border,
and using a Monte-Carlo approximation of the integrals.
The optimization problem \eqref{formBT} can be replaced by the
generalized eigenvalue problem
\begin{equation}
\label{formBTdiscr}
\lambda\mathbf F\mathbf u = \mathbf G\mathbf u
\end{equation}
for which the largest eigenvalue is the inverse of $\tau(k)$. Note that here, the size of the matrices is the size of the
approximating family, and does not depend on the number of samples used in the 
domain and on its boundary.

The eigenfrequencies are found as the local minima of $\tau(k)$.
This method offers additional advantages. The coefficients
of the expansion of the eigenmodes are readily available
as the first eigenvector of Eqn. \eqref{formBTdiscr}, and
$n$-multiple eigenfrequencies are characterized by the fact that
the $n$ first eigenvalues of Eqn. \eqref{formBTdiscr} exhibit a local minimum.
Here, a basis of the eigenspace is obtained using the $n$
first eigenvectors of Eqn. \eqref{formBTdiscr}. Note that this basis
has no reason to be orthogonal, as it is the vectors of the coefficients
of the expansions  of the basis functions that are orthogonal, not the functions themselves.

The same method, with a slightly different implementation, was used
by Barnett and Berry \cite{barnett} to compute high frequency modes of quantum
cavities. As they considered only domains with smooth boundaries,
plane wave families were sufficient for the application of the method.

To generalize this method to plates, two adaptations are needed:
\begin{itemize}
\item an approximation scheme for solutions of Eqn. (\ref{bilap})
has to be developed
\item the boundary conditions encountered in plate problems
have to be modeled in a way compatible with formulation Eqn. (\ref{formBT}).
\end{itemize}
These two points are the topics of the next two sections.

\section{Approximation of plate eigenmodes}
\label{sec:app}

In this section, we prove that solutions of Eqn. (\ref{bilap}), with
arbitrary boundary conditions, can be approximated by sums of 
Fourier-Bessel functions and modified Fourier-Bessel functions. We first
give a short account of the Vekua theory for the Laplace operator,
and then extend these results to the bi-Laplace operator.

The domain where the functions 
are defined
is assumed to be star-shaped, and to contain the ball centered on the origin
of radius $\rho h$, where $h$ is the diameter of the domain. Furthermore,
we assume that the domain satisfies the exterior cone condition with angle $\beta \pi$. This means that each point of the border is the vertex of a cone
of angle  $\beta \pi$ which does not intersect the interior of the domain.
Such a domain is pictured Figure \ref{domain}.

\begin{figure}
\center
\includegraphics[width=8cm]{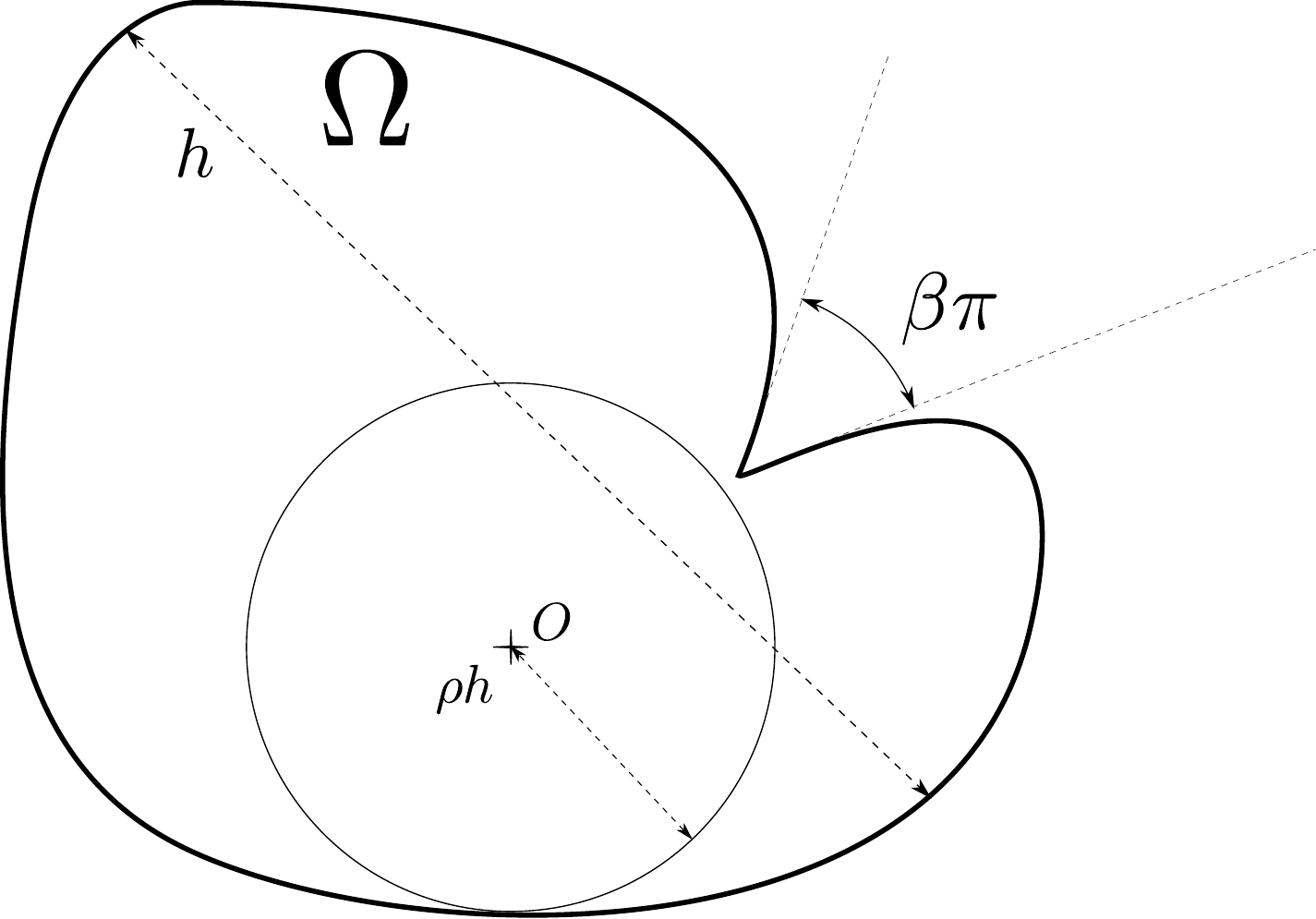}
\caption{A star-shaped domain satisfying the exterior cone condition with
angle $\beta \pi$}
\label{domain}
\end{figure}

In the following, we give the approximation bounds  in weighted Sobolev norms defined by
\begin{equation}
\label{norm}
\| u \|^2_{m,k} = 
\sum_{p=0}^m \frac{1}{k^{2p}} \sum_{p_1 + p_2 = p} \int_{\Omega} \left|\frac{\partial^p}{\partial x^{p_1}\partial y^{p_2}}
 u\right|^2 \!\!\!dxdy
\end{equation} 

\subsection{Vekua theory for the Laplace operator}

A simple example of an approximation of a solution to a differential equation
is the case of holomorphic functions. These functions, solutions
to the Cauchy-Riemann equations, can be, on a simply connected domain,
approximated by polynomials of the complex variable. By taking the real
part of a holomorphic function and of its approximations, it is shown
that we can approximate a harmonic function, solution to $\Delta u = 0$,
by harmonic polynomials in $\mathbf R^2$.

The Vekua theory, exposed in \cite{vekua} and summarized in
\cite{henricivekua}, gives similar results for solutions of elliptic
PDEs by generalizing the operation ``taking the real part'',
which allows us to map holomorphic functions to harmonic functions,
to solutions of these PDEs. Using
these operators, which are continuous and continuously invertible, approximation of holomorphic functions
by polynomials of the complex variable is translated to
approximation of solutions to the PDEs by the images of polynomials. 
Further details
on these operators in the case of the Helmholtz equation
\begin{equation}
\Delta u + \lambda u = 0
\label{helm}
\end{equation} 
and their properties in Sobolev spaces can be found in \cite{moiolavekua}.

The main result of Moiola \textit{et al.} is that the solutions of Eqn. \eqref{helm} can be approximated by 
generalized harmonic polynomials, \textit{i.e.} functions of the form
\begin{equation}
u_L = \sum_{n = -L}^{L} \alpha_n J_n(kr) e^{in\theta}
\end{equation}
where $(r,\theta)$ are the polar coordinates and $k = \sqrt{\lambda}$.

When $\lambda$ is strictly positive, Moiola \textit{et al.} \cite{moiolavekua} 
have shown that there exists a generalized
harmonic polynomial of order at most $L$ such that
\begin{equation}
\begin{split}
\| u - Q_L\|_{j,k}  \leq  C (1+ (k h)^{j+6}) e^{3(1-\rho)k h/4}\\
\left(
\frac{\log (L+2)}{L+2}\right)^{\beta(K+1-j)}
\|u\|_{K+1,k}
\end{split}
\end{equation}
In the case where $\lambda < 0$, the bound is multiplied by $e^{3k h/2}$.

When the function to be approximated is infinitely differentiable in an open domain containing 
$\Omega$,
the convergence is exponential in $L$ \cite{melenk}.

Moiola \textit{et al.} used these results to analyze the approximation of solutions of Eqn. (\ref{helm}) by sums of plane
waves \cite{moiola2}.  The orders of convergence are identical to the generalized harmonic polynomials case.

\subsection{Extension to plates eigenmodes}

Eigenmodes of a homogenous plate of rigidity $D$, specific mass  $\rho$ and
thickness $h$, with in-plane tension $T$, are solutions of
\begin{equation}
\label{eq_plaque}
D\Delta^2 w + T \Delta w - \rho h \omega^2 w = 0.
\end{equation}

In order to approximate such functions, we reduce this problem
to the approximation of two solutions of the Helmholtz equation.
Indeed, solutions of equation \eqref{eq_plaque} can be decomposed as a sum
of two solutions of equation \eqref{helm}, with parameters deduced from the
properties of the plate.

\begin{lem}
\label{dec}
Let $w$ be a solution to (\ref{eq_plaque}) in the sense of distributions.
Then $w$ can be decomposed as the sum of $w_1$ solution of $\Delta w_1 - \lambda_1 w_1 = 0$,
and $w_2$ solution of $\Delta w_2 - \lambda_2 w_2 = 0$, where $\lambda_1$ and $\lambda_2$ 
are the zeros of

\begin{equation}
\label{poly}
D\lambda^2 + T\lambda - \rho h \omega^2.
\end{equation}

Furthermore, if $w\in H^{K+2}$, then 
\begin{align}
||w_1||_{K, k_1} & \leq  \frac{2k_+^2}{\delta_\lambda}||w||_{K+2,k_1} \label{maj1}\\
||w_2||_{K, k_2} & \leq  \frac{2k_+^2}{\delta_\lambda}||w||_{K+2,k_2}. \label{maj2}
\end{align}
where $k_1 = \sqrt{|\lambda_1|}$ and $k_2 = \sqrt{|\lambda_2|}$, $k_+$ being the largest, and $\delta_\lambda = \sqrt{T^2 + 4D\rho h \omega^2}/D$ the difference between
$\lambda_1$ and $\lambda_2$.
\end{lem}

\begin{proof}
Analysis : assuming the decomposition, we find
\begin{align}
\Delta w - \lambda_2 w & =  (\Delta w_1 - \lambda_2 w_1)
+ (\Delta w_2 - \lambda_2 w_2)\\
 & = (\lambda_1 - \lambda_2) w_1 + 0
\end{align}
A similar computation for $w_2$ gives
$$w_1 = \frac{1}{\lambda_1 - \lambda_2} (\Delta w - \lambda_2 w),
\quad
w_2 = \frac{1}{\lambda_2 - \lambda_1} (\Delta w - \lambda_1 w).
$$
Note that $\lambda_1$ and $\lambda_2$ are always distinct as $\delta_\lambda$
is always strictly positive.

Synthesis : we check that $w = w_1  + w_2$ :
\begin{align}
w_1 + w_2 & =  \frac{1}{ \lambda_1 - \lambda_2} (\Delta w - \lambda_2 w)
- (\Delta w - \lambda_1 w) \\
& = w
\end{align}
Then that $\Delta w_1 - \lambda_1 w_1 = 0$ :
\begin{align}
\Delta w_1  - \lambda_1 w_1 & =  
\frac{(\Delta^2 w - \lambda_2 \Delta w) - (\lambda_1 \Delta w - \lambda_1 \lambda_2 w)}{\lambda_1 - \lambda_2}
 \\
& =  \frac{\Delta^2 w - (\lambda_2 + \lambda_1) \Delta w + \lambda_1 \lambda_2 w}{\lambda_1 - \lambda_2}\\
& =  0
\end{align}
The last equality comes from the fact that $\lambda_1$ and $\lambda_2$ are the zeros of the polynomial
$D\lambda^2 + T \lambda -\rho h\omega^2$. We also find that $\Delta w_2 -\lambda_2 w_2 = 0$.

Finally, if $w \in H^{K+2}$, then
\begin{align}
\|w_1\|_{K,k_1} & \leq  \frac{1}{|\lambda_1 - \lambda_2|}
\left(\|\Delta w \|_{K,k_1} + \lambda_2 \| w \|_{K,k_1} \right)\\
& \leq  \frac{1}{\delta_\lambda} \left( k_1^2 \| w\|_{K+2,k_1} + k_2^2 \| w \|_{K+2,k_1}\right)\\
& \leq  \frac{2k_+^2}{\delta_\lambda} \| w\|_{K+2,k_1}
\end{align}
The result is identical for $w_2$.
\end{proof}

\begin{remarque}
\label{remrem}
In the case where no tension is applied to the plate, \textit{i.e.} $T = 0$, 
we have $k_1 = k_2 = (\rho h /D)^{1/4} \sqrt{\omega} = k$, $\delta_\lambda = 2k^2$ and
the results can be simplified as :
$$||w_1||_{K,k} \leq ||w||_{K+2,k}, \ \ \ ||w_2||_{K,k} \leq ||w||_{K+2,k}.$$
\end{remarque}

\begin{remarque}
Two orders are lost in the majorizations \eqref{maj1} and \eqref{maj2} : 
the norm of order $K$ of both components of $w$ are bounded by their norm of order $K+2$.
It is impossible, in the general case, to have a better bound. Let us consider, on
a disk sector centered at the origin, the function $f$ defined by $f = e^{i\theta/2}(J_{1/2}(r) - I_{1/2}(r))$
in polar coordinates. This function is solution of  $\Delta^2 w - w = 0$, and can be decomposed as the sum
of $f_1 = e^{i\theta/2}J_{1/2}(r)$, solution of $\Delta f_1 + f_1 = 0$, and $f_2 = - e^{i\theta/2}I_{1/2}(r)$, solution
of $\Delta f_2 - f_2 = 0$. These two functions have a radial behavior at the origin similar to
 $r^{1/2}$, and thus are not in $H^3$. Their sum however behaves like $r^{5/2}$, and is in $H^4$.
\end{remarque}

If $\lambda_1$, or $\lambda_2$, is negative,
the corresponding component of $w$ can be readily approximated
by generalized harmonic polynomials or plane waves
using the results of Moiola \textit{et al.} \cite{moiola2}.
If $\lambda_1$ or $\lambda_2$ is positive, the associated component
can be approximated in a similar way. In that case,
we use either a family of modified Fourier-Bessel functions;
where the Bessel functions $J_n$ are replaced by modified Bessel functions
$I_n$, or a family of exponential functions  $e^{\mathbf k \cdot \mathbf x}$
instead of plane waves. The bounds on the approximation error are similar.

\begin{theo}
\label{app_kl}
Let $\Omega$ be a domain satisfying the assumptions of this section, $K \geq 1$ integer,
and $w \in H^{K+2}$ verifying conditions of lemma \ref{dec}.
Then for all $L > K$,
there exist two generalized harmonic polynomials
$P_L$ and $Q_L$ with parameters $\lambda_1$ and $\lambda_2$,  
of degree at most $L$
such that for all $j \leq K$,
\begin{equation}
\begin{split}
\|w - (P_L + Q_L)\|_{j,k_+}
\leq C \frac{k_+^2}{\delta_\lambda}(1+(k_+ h)^{j+6}) e^{\frac{3}{4}(3-\rho)k_+h}\\
 \left( \frac{\ln (L+2)}{L+2}
\right)^{\beta(K-j)} (k_+h)^{K-j}\| w\|_{K+2,k_-},
\end{split}
\end{equation}
where $k_-$ is the smallest of $k_1$ and $k_2$, and $k_+$ the largest,
{and $C$ is a constant which depends only on the shape of $\Omega$, $j$ and $K$}. If $\lambda_1$ and $\lambda_2$
are both negative, the bound can be divided by $e^{3/2 k_+ h}.$
\end{theo}

\begin{proof}
The first step of the proof is to decompose $w$ using lemma \ref{dec}.
Using Theorem 2.2.1.ii and Remark 1.2.6 from \cite{moiola}, 
we can approximate these two components by generalized harmonic polynomials,
modified if $\lambda$ is positive.

Let us assume $\lambda_1 < 0$, $\lambda_2 >0$, with $|\lambda_2| > |\lambda_1|$.
Other cases can be treated similarly.

We have
\begin{align*}
\|w - (P_L + Q_L)\|_{j,k_+} & \leq  \|w_1\|_{j,k_+} + \|w_2\|_{j,k_+} \\
& \leq  \|w_1\|_{j,k_1} + \|w_2\|_{j,k_2} \\
\begin{split}
& \leq  C (1+(k_1 h)^{j+6}) e^{\frac{3}{4}(1-\rho)k_1h} \\
& \left( \frac{\ln (L+2)}{L+2}
\right)^{\beta(K-j)} (k_1h)^{K-j}\| w_1\|_{K,k_1} \\
&  + C (1+(k_2 h)^{j+6}) e^{\frac{3}{4}(3-\rho)k_2h} \\
& \left( \frac{\ln (L+2)}{L+2}
\right)^{\beta(K-j)} (k_2h)^{K-j}\| w_1\|_{K,k_2}
\end{split}\\
\begin{split}
& \leq  C \frac{k_+^2}{\delta_\lambda}(1+(k_+ h)^{j+6}) e^{\frac{3}{4}(3-\rho)k_+h} \\
&\left( \frac{\ln (L+2)}{L+2}
\right)^{\beta(K-j)} (k_+h)^{K-j}\| w\|_{K+2,k_-}
\end{split}
\end{align*}
\end{proof}

\begin{remarque}
For plates without in-plane tension, the result is slightly simpler.
In that case, the roots of Eqn. (\ref{poly}) have same absolute values and opposite signs.
We thus have
\begin{align}
\begin{split}
\|w - (P_L + Q_L)\|_{j,k} &
\leq C (1+(k h)^{j+6}) e^{\frac{3}{4}(3-\rho)kh} \\
& \left( \frac{\ln (L+2)}{L+2}
\right)^{\beta(K-j)} (kh)^{K-j}\| w\|_{K+2,k}.
\end{split}
\end{align}
\end{remarque}
 
\begin{remarque}
We prove the approximation result for sums of Fourier-Bessel functions.
However, as Fourier-Bessel functions can be approximated by sums of plane
waves, modified Fourier-Bessel functions can be approximated by sums
of exponential function of the type $e^{\vec k \cdot \vec x}$ with constant
norm of $\vec k$. This allows an approximation of solutions of Eqn. \eqref{eq_plaque}
by sums of plane waves and exponential functions.
\end{remarque}

\section{Boundary conditions}
\label{sec:form}

Boundary conditions usually encountered in plate problems
(clamped edges, simply supported edges and free edges) are not
as such readily usable for the numerical scheme proposed here, and
thus need to be modeled differently.

For clamped edges, the displacement and its normal derivative
are zero. The tension for clamped boundary conditions is then:
$$t_c = \int_\Gamma \left| w \right|^2 + \frac{1}{k^2}\int_\Gamma \left|
\dpart{w}{n} \right|^2.$$
A simply supported edge has zero displacement and 
torsion moment $M_n$
$$t_s = \int_\Gamma \left| w \right|^2 + \frac{1}{D^2k^4}\int_\Gamma \left|M_n\right|^2.$$
Finally, free edges have zero torsion moment and Kelvin-Kirchhoff edge reaction $K_n$
$$t_l = \frac{1}{D^2k^4}\int_\Gamma \left|M_n\right|^2
+ \frac{1}{D^2k^6}\int_\Gamma \left|K_n\right|^2.$$
The torsion moment and the Kelvin-Kirchhoff edge reaction writes \cite{geradin}
$$M_n = -D \left( \dpart[2]{w}{n} + \nu \dpart[2]{w}{t} \right)$$
$$K_n = -D \left(\dpart[3]{w}{n} + (2-\nu) \frac{\partial^3 w}{\partial
n \partial t^2} + \frac{1- \nu}{R}
\left( \dpart[2]{w}{n} - \dpart[2]{w}{t} \right)\right)$$
where $\nu$ is the Poisson ratio of the material, and
$R$ the curvature radius of the boundary.

The constants $1/k^2$, $1/D^2k^4$ and $1/D^2k^6$ in front of some integrals
not only make the quantities homogeneous, but also improve the
numerical stability, as the contribution of the two integrals
are rescaled to have the same order of magnitude. For instance,
in the case of $t_c$ estimated using plane waves, the first term contains products of planes waves, while the second term contains products of derivatives of plane waves, which are the plane waves themselves multiplied by the scalar product of
the wave vector and a unit vector normal to the boundary. With no rescaling,
the second term would be more and more influent as the frequency increases.
This rescaling is similar to the weights used to define the norms
$\| \cdot \|_{m,k}$ in Eqn. \eqref{norm}.

Note that for a plate with various boundary conditions along the border,
the corresponding tension is the sum of the tension for the different
boundary conditions, integrated on their respective domain.

\section{Numerical results}
\label{sec:num}

We now compare the methods with analytical results for simple cases, and
with numerical results obtained with Cast3M  \cite{castem}, a 
widely-used FEM simulation program.
To avoid technicalities, the simulated plates are star-shaped with smooth
boundaries. The treatment of more general shapes is discussed in the next section.
Since these numerical tests are done without in-plane tension, we can use
the wavenumber $k=(\rho / D)^{1/4} \sqrt{\omega}$ to express the eigenfrequencies. In the case of in-plane tension, the wavenumbers used to generate the 
Fourier-Bessel functions (resp. plane waves) and modified Fourier-Bessel
functions (resp. exponential functions) are computed according to Lemma
\ref{dec}.

We first test the method on a circular plate of radius 1, with various boundary conditions. 
Here, we use plane
waves, as eigenmodes of circular plates are sums of a Fourier-Bessel function
and a modified Fourier-Bessel function, making the numerical problem
trivial. Table \ref{tabeg} gives the computed eigenfrequencies
and the values given in Leissa  \cite{leissa} for low-frequency modes. Figure \ref{tens1} shows the tension
as a  function of the frequency in the clamped case, and the
next  three eigenvalues of the numerical problem \eqref{formBTdiscr}. These can be 
used to identify multiple eigenvalues and to compute a basis of the
associated eigenspaces. Examples of eigenmodes are given
for the three boundary conditions in Figure \ref{em1}. In these numerical experiments,
the boundary of the disk was discretized with 2048 points, its interior
with 1024 points drawn using the uniform distribution on $\Omega$, and the
 number of plane waves (which is also the size of the numerical problems to solve) was 
 taken as $8k$, which in this case ranges from 28 to 80.

\begin{figure}
\centering
\includegraphics[width=9cm]{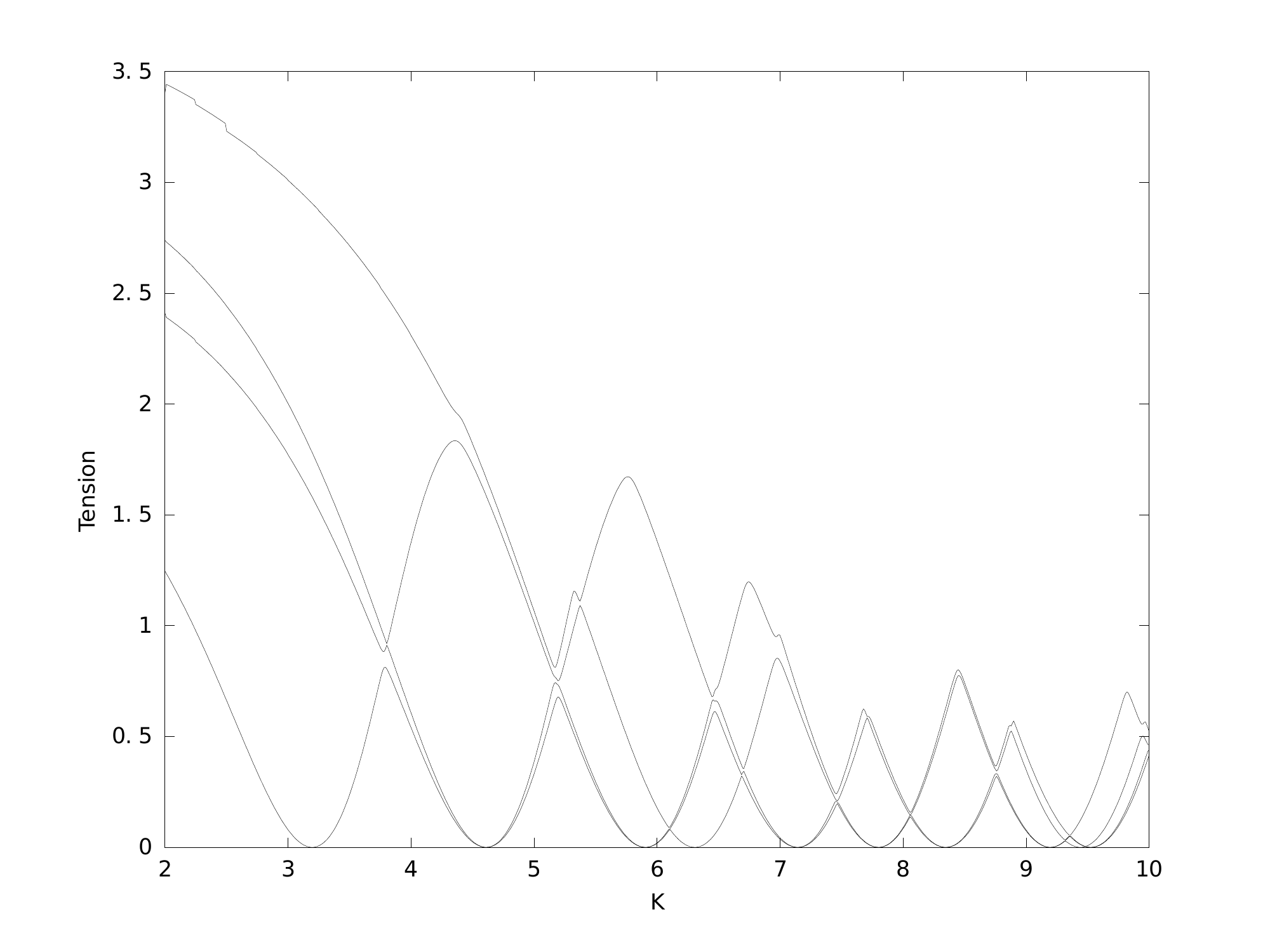}
\caption{Inverses of the four largest eigenvalues (the first one being the tension) of discrete
problem \eqref{formBTdiscr}, for the clamped circular plate}
\label{tens1}
\end{figure}

\begin{figure*}
\centering
\includegraphics[width=15cm]{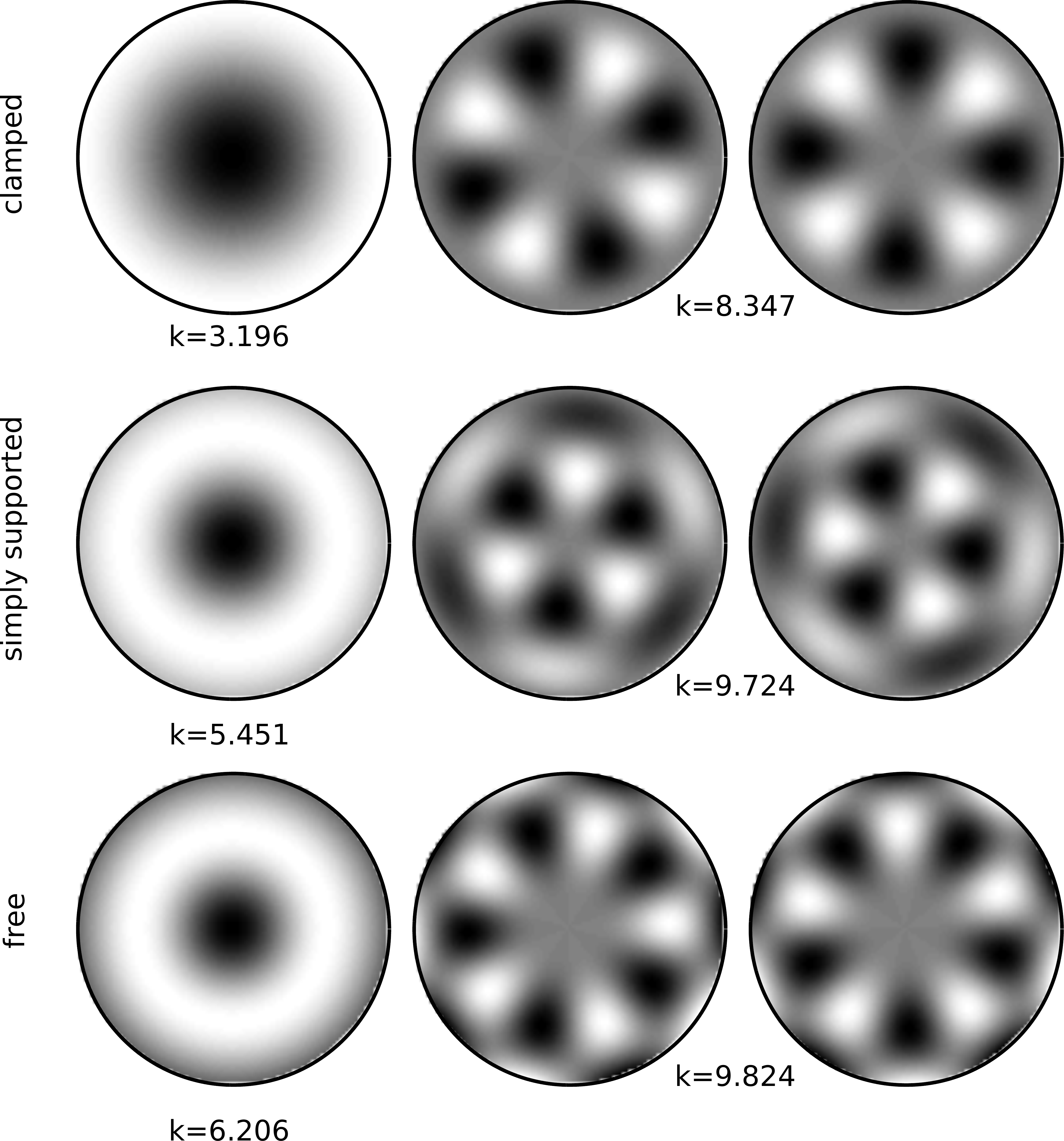}
\caption{Some examples of eigenmodes, for simple and double eigenvalues,
of the circular plate with clamped, simply supported and free boundary conditions}
\label{em1}
\end{figure*}

\begin{table}
\begin{tabular}{c|cc|cc|cc}
& \multicolumn{2}{c|}{Clamped} &
\multicolumn{2}{c|}{Simply supported} &
\multicolumn{2}{c}{Free}\\
& Leissa & MPS-PW & Leissa & MPS-PW & Leissa & MPS-PW \\
\hline
1 & 3.196 & 3.196 & 2.23 & 2.22 & 2.29 & 2.29 \\
2 & 4.611 & 4.611 & 3.73 & 3.73 & 3.01 & 3.01 \\
3 & 5.906 & 5.905 & 5.06 & 5.06 & 3.50 & 3.50 \\
4 & 6.306 & 6.306 & 5.46 & 5.45 & 4.53 & 4.53 \\
5 & 7.144 & 7.144 & n/a & 6.32 & 4.65 & 4.64 \\
\end{tabular}
\caption{Low eigenfrequencies for a circular plate with various boundary conditions, in Leissa \cite{leissa} and computed with the MPS and plane waves (MPS-PW).}
\label{tabeg}
\end{table}

\medskip

We now compute eigenfrequencies and eigenmodes of a second plate, of more complex shape, with boundaries defined by the
parametric equations
\begin{equation}
\left\{
\begin{array}{l}
x= \cos t \\
y= \sin t + \frac{1}{3}\sin 2t
\end{array}
\right.
\ \ \ \ t \in [0, 2\pi)
\end{equation}

Table \ref{tabeg2} gives the first ten eigenfrequencies of this plate 
with the three types of boundary conditions : clamped, simply supported and free. 
The proposed method, MPS with plane waves (MPS-FB), uses 
2048 points on the border and 1024 points inside.
The number of plane waves is $10k$, which here ranges from 20 to 80.
For the clamped conditions, the results using Fourier-Bessel functions
are also given (method MPS-FB). 
We compare our results with those obtained by the FEM, as 
implemented in the Cast3M package {using a triangular mesh
and finite elements of order 3, with 6926 elements} and the 
border discretized by 180 
segments.

For two eigenmodes of the clamped plate,
we give  in table \ref{tabeg22} the estimated eigenfrequencies with varying size of the discrete problems
for FEM and the MPS-PW, and compare them to the values
obtained in \cite{alves} with the MFS (with an unspecified size).
These two modes, computed with the MPS-PW, are shown on Figure \ref{figem2}.

\begin{table*}
\centering
\begin{tabular}{c|ccc|cc|cc}
&\multicolumn{3}{c|}{Clamped} &\multicolumn{2}{c|}{Simply supp.}&\multicolumn{2}{c}{Free} \\
 & FEM & MPS-PW & MPS-FB & FEM & MPS-PW & FEM & MPS-PW \\
\hline
1 & 3.3786 & 3.3782 & 3.3562 & 2.4104 & 2.4061 & 2.1933 & 2.2000 \\
2 & 4.5937 & 4.5944 & 4.5864 & 3.6727 & 3.6733 & 2.3516 & 2.4001 \\
3 & 4.6267 & 4.9275 & 4.8815 & 4.0065 & 3.9993 & 2.7468 & 2.7481 \\
4 & 5.8359 & 5.8376 & 5.8366 & 4.9500 & 4.9505 & 3.4225 & 3.4222 \\
5 & 6.2048 & 6.2067 & 6.1987 & 5.3203 & 5.3216 & 3.7624 & 3.7623 \\
6 & 6.4551 & 6.4567 & 6.4527 & 5.5451 & 5.5456 & 4.0996 & 4.0993 \\
7 & 7.0716 & 7.0738 & 7.0738 & 6.2134 & 6.2147 & 4.2007 & 4.2004 \\
8 & 7.4914 & 7.4949 & 7.4949 & 6.6169 & 6.6188 & 4.7986 & 4.7985 \\
9 & 7.7709 & 7.7740 & 7.7590 & 6.9041 & 6.9058 & 4.9847 & 4.9845 \\
10 & 7.9676 & 7.9300 & 7.9660 & 7.0629 & 7.0308 & 5.4319 & 5.4316 
\end{tabular}
\caption{First ten eigenfrequencies, expressed in wavenumber $k$,  of
the second plate with various boundary conditions, for FEM, MPS with plane waves (MPS-PW) and
MPS with Fourier-Bessel functions (MPS-FB, clamped conditions only)}
\label{tabeg2}
\end{table*}

\begin{table*}
\centering
\begin{tabular}{ccccc|c|ccc}
\multicolumn{5}{c|}{FEM} & MFS
&  \multicolumn{3}{c}{MPS-PW}\\
\hline
\multicolumn{5}{c|}{number of elements} &
&  \multicolumn{3}{c}{number of plane waves}\\

56 & 242 & 1048 & 6926 & 28171 &  & 50 & 60 & 70\\\hline
8.7505 & 8.9527 & 9.0714& 9.1066 & 9.1111 & 9.11259
& 9.1068 & 9.1121 & 9.1126 \\
9.0159 & 9.1555 & 9.2438 & 9.2735 & 9.2777 & 9.27903
& 9.2651 & 9.2722 & 9.2787
\end{tabular}
\caption{Eigenfrequencies for two modes of the clamped second plate, with various
discretization sizes, computed by FEM (as implemented in Cast3M), {MFS (results of \cite{alves} on the same plate), and MPS-PW, with various number of plane waves.}}
\label{tabeg22}
\end{table*}

\begin{figure}
\includegraphics[width=8cm]{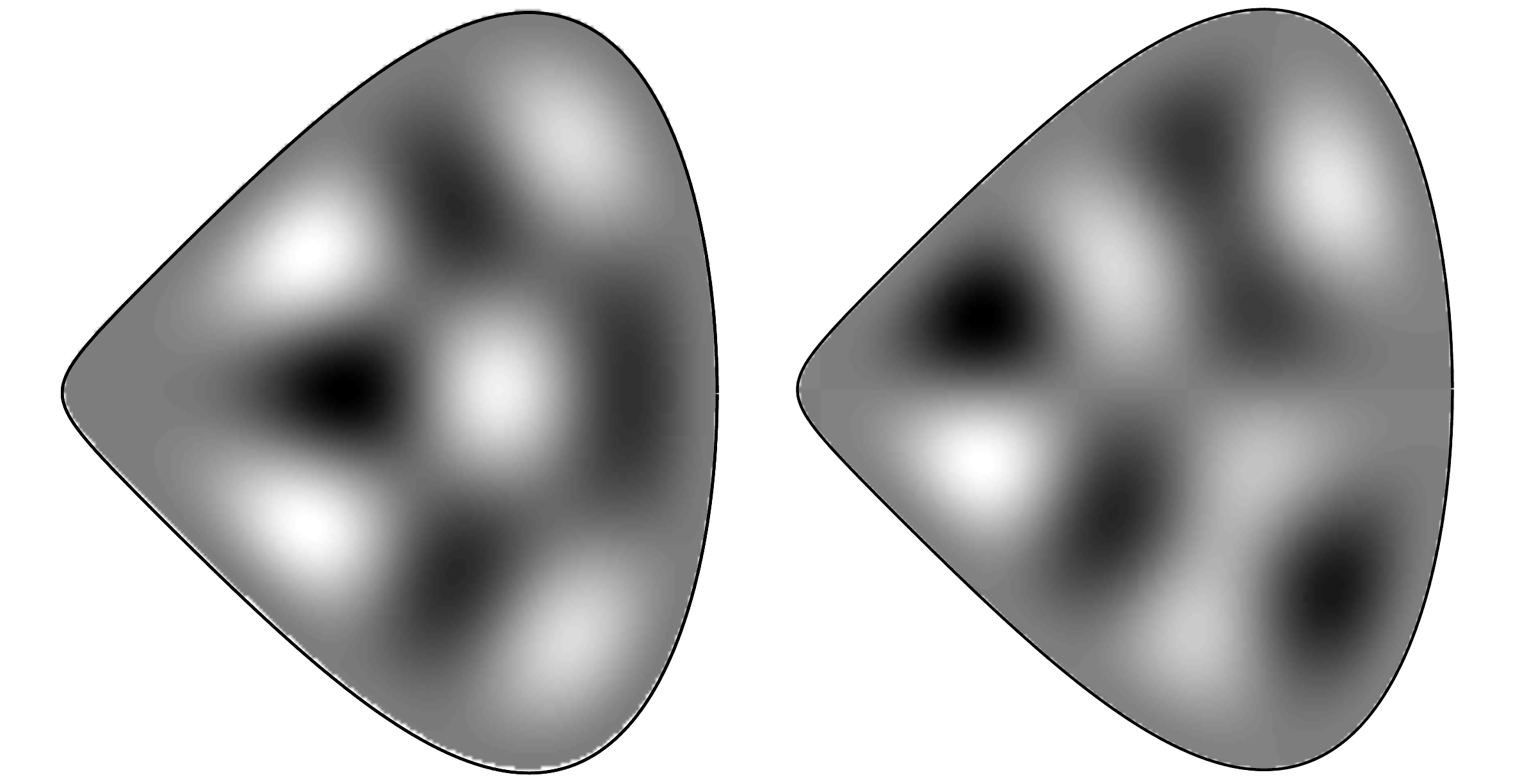}
\caption{Eigenmodes of the second plate, for  $k=9.1126$ (left) and $k=9.2787$ (right)
computed by MFS-PW, for the clamped boundary conditions}
\label{figem2}
\end{figure}

\section{Discussion}
\label{sec:discus}

We now discuss some implementation details, relative to the treatment
of more complex shapes, or the acceleration of the computations,
and the relative merits of the MPS and other methods.

\subsection{Fourier-Bessel vs. plane waves}
\label{subsec:fbpw}

As shown by the numerical experiments, both Fourier-Bessel functions
and plane waves can be used to approximate eigenmodes. They have similar
approximation properties, but differ implementation-wise.
Fourier-Bessel functions are orthogonal on a disc, ensuring better stability,
while plane waves are more and more ill-conditioned as their number
increases. However, this can be treated by pre-conditioning the
plane waves family with a discrete Fourier Transform, mapping the
plane waves to approximations of the Fourier-Bessel functions.

The main advantage of the plane waves is the straightforward computation
of their derivatives : differentiating a plane wave along a certain
direction amounts to multiplying it with the scalar product of its
wave vector with a unit vector. This makes the construction
of the matrices both easy to implement and fast.

\subsection{Shapes}

The proposed method relies on an approximation scheme for solutions
to the equation \eqref{bilap}, and therefore is limited to cases where
such approximations are available. In particular, the approximation
scheme assumes a star-shaped domain.
In the case of a simply connected, but non-star convex, domain, the 
approximation is not guaranteed to succeed. A possible way to overcome
the problem is to cut the domain into star-convex subdomains, to approximate
the solutions of \eqref{bilap} in these subdomain, and to add
terms in the tension ensuring that the displacement, normal
derivative of the displacement, bending moment and strain have
the same value at both sides of the internal boundaries. A similar
method as already been applied to the particular case of polygonal membranes
\cite{descloux}.

Other cases of non-convex domains are domains with holes, but star-convex
if the holes are filled. In that case, following Vekua, one can approximate
solutions of \eqref{bilap} by adding to the family of planes waves or Fourier-Bessel 
functions, the sets of Fourier-Bessel functions of second kind $e^{in\theta} Y_n(kr)$
and $e^{in\theta} K_n(kr)$, one for each hole, centered on a point chosen
in each of them. This type of approximation can be compared to the
approximation of holomorphic functions given by the Runge theorem.
An application to membranes can be found in \cite{betckethese}.

\subsection{Singularities}

The domains considered here have smooth boundaries. This guarantees that
the eigenmodes of the plates are smooth, and that the convergence of
the approximations is fast. However, as shown in Theorem \ref{app_kl}, singularities, which can appear at corners
of a domain with non-smooth boundaries, slow down the convergence of the approximations. The size of the discretization for such cases is then larger
than for smooth boundaries, possibly too large to guarantee the
numerical stability of the computations. To accelerate the convergence,
FHM, Eisenstat and BT use fractional Fourier-Bessel functions centered on the
singular corner of the domain (in their numerical experiments, an
L-shaped polygon).  This idea has been used for plates in a different setting
by De Smet et al. \cite{desmetsing}.

Note that, for a plate with polygonal holes, at least a corner of the
hole is singular. In that case, it is impossible to use fractional
Fourier-Bessel functions, as it is impossible to define them on a domain
containing a path around the origin. In that case, one can combine
fractional Fourier-Bessel function with the method
described in the previous section \cite{betckethese}.

\subsection{Numerical considerations}

Although the numerical stability is improved compared to the determinant
based MPS, the improved version is still prone to instabilities when
a large approximation order is used. These instabilities are further
amplified in the case of plates, where modified Fourier-Bessel functions,
or exponential functions, are included in the approximating families.
The behavior of these functions are such that they are non-negligible only
on a small region near the boundary of the domain. Using the Monte-Carlo
approximation with uniform density to estimate the coefficients of the matrices is thus unstable. Using a non-uniform density of samples, with more
samples near the border would be a way to improve the stability of the
estimations, in a way similar to what is used in \cite{acha}, where the
reconstruction of a solution of the Helmholtz equation on a disc is improved by placing a fraction of the samples on the border of the disc.

\subsection{Speeding up the eigenvalue search}

In order to locate the minima of the tension, the proposed algorithm simply 
computed it on linearly spaced
values in the interval we were interested in. However, the particular
behavior of the tension (and more generally, of the eigenvalues of
problem \eqref{formBTdiscr}) could be used to accelerate the search
of these minima. Indeed, the tension is a sequence of branches which
behave more or less as parabolas. Newton iterations, along with the
computation of the derivatives of the eigenvalues of problem \eqref{formBTdiscr}, 
could therefore be used to quickly locate the minima.

\subsection{Comparison with other methods}

{The simulations (Table \ref{tabeg22}) show that the MPS method needs a significantly lower order
of approximation
than the FEM to obtain accurate results. In particular,
the MPS using 50 Fourier-Bessel functions achieve a much better estimation
of the eigenfrequencies that the FEM with 56 elements. Indeed, while the FEM needs elements of size comparable with
the tenth of the wavelength (and thus an order scaling like the square of the wavenumber),
the MPS needs an order proportional to the wavenumber \cite{barnett}, because only
the solutions to Eqn.~(\ref{bilap}) are approximated.}

{The computation times for the FEM and the MPS to obtain Table \ref{tabeg2}
were similar. However the MPS was here implemented using Octave and Matlab,
without the improvements given in the previous subsection,
and significants gain in term of complexity can be expected.
The implementation of the MPS is also simpler,
as no meshing is required.}

{The MFS, like the MPS, needs low orders to achieve correct results.
However, to use the MFS (or the NDIF), one has to choose a set of charge points. The
accuracy of the method can vary with this choice, and while
methods have been suggested for this task \cite{alves}, an
optimal and efficient way is not yet known. Here, in the case of the 
plane wave approximation, no choice is needed. Furthermore, 
the use of Bessel functions makes the implementation delicate,
as pointed out in Subsection \ref{subsec:fbpw}.}

{Singularities can be treated  with FEM methods:
using adaptive methods, such as $hp$-FEM  \cite{Babuska}, a fast convergence in presence of singularities
can be achieved. The treatment
of singularities is however simpler using in the MPS, as the 
singular functions to be used can be determined a priori directly from
the geometry of the problem.}

{The solutions given by the MPS are smooth, as they are finite sums of
smooth functions. Galerkin methods, such as the Element Free Galerkin method, have also been developed to improve the regularity of the solutions, which is limited
in the FEM by the regularity of the elements. However, as pointed out
in the case of the FEM,
the fact that these methods approximate a functional space
not limited to the solutions to Eqn.~(\ref{bilap}) leads to a large
order of approximation to achieve good results.}

{While the MPS has significant advantages compared to other methods,
the need of a family approximating the solutions to Eqn.~(\ref{bilap}) 
restricts its use to cases where such a family is known.
Problems with plates of heterogeneous materials or
thickness can be difficult or impossible to treat using this method.}

\section{Conclusion}

This paper has described the extension of the Method of Particular Solutions for
the computation of eigenmodes of plates. This method has numerous
advantages. It can be used with any approximation scheme for the solutions
of the studied equation ; in this paper, we used Fourier-Bessel functions and plane waves,
but the method could be extended with fractional Fourier-Bessel functions in order to treat
singularities. Its formulation offers a large flexibility in the discretization 
of the domain,  and independently, in the size of the numerical problem. 
The determination of multiple eigenvalues and eigenmodes
is also straightforward.  Finally, the so-called tension, that has to 
be minimized to find the eigenfrequencies, has a specific 
shape that  can be used to speed up the search.
Future improvements include sampling schemes
yielding better stability of the numerical problems, and {more efficient computational methods}
for high frequency eigenmodes.

\section*{Reproducible research}

The Matlab/Octave code used to compute eigenfrequencies and eigenmodes of Figures \ref{em1} and \ref{figem2}
is available online \cite{code}.

\section*{Acknowledgement}
The authors
acknowledge partial support from Agence Nationale de la Recherche (ANR), project
ECHANGE (ANR-08-EMER-006), project LABEX WIFI (ANR-10-IDEX-0001-02 PSL*),
and Austrian Science Fund (FWF) START-project FLAME (Y 551-N1).

\bibliographystyle{ieeetr}
\bibliography{modal}

\end{document}